\documentclass[10pt]{amsart}

\usepackage[english]{babel}
\usepackage[T1]{fontenc}
\usepackage{tikz}
\usepackage{amsmath,amssymb,stmaryrd,amsthm}
\usepackage[all]{xy}

\usepackage{fancyhdr}

 \theoremstyle{plain}
 \newtheorem{lemma}{Lemma}[section]
\newtheorem{theorem}[lemma]
{Theorem }
\newtheorem{corollary}[lemma]
{Corollary}

\newtheorem*{thma}{Theorem 1.1'}
\newtheorem*{thmb}{Theorem 1.2'}

\theoremstyle{definition}
\newtheorem{definition}[lemma]{Definition }

\newtheorem{ex}[lemma]
{Example }

\newtheorem{rmk}[lemma]{Remark}

\newcommand{\lgw}{\longrightarrow}
\newcommand{\lgm}{\longmapsto}

\newcommand{\ovl}{\overline}

\renewcommand{\deg}{\text{deg}\,}
\newcommand{\ord}{\text{ord}}

\newcommand{\wdt}{\widetilde}

\renewcommand{\O}{\Omega}

\newcommand{\m}{\mathfrak{m}}

\newcommand{\Z}{\mathbb{Z}}

\renewcommand{\k}{\Bbbk}
\newcommand{\Spec}{\text{Spec}}

\newcommand{\R}{\mathbb{R}}

\newcommand{\N}{\mathbb{N}}

\newcommand{\Q}{\mathbb{Q}}

\newcommand{\D}{\Delta}
\newcommand{\lb}{\llbracket}
\newcommand{\rb}{\rrbracket}

\renewcommand{\a}{\alpha}
\renewcommand{\b}{\beta}
\newcommand{\g}{\gamma}

\renewcommand{\phi}{\varphi}
\renewcommand{\d}{\delta}

\begin{document}
\title[\L ojasiewicz inequality]{\L ojasiewicz inequality over the ring of power series in two variables}
\author{Guillaume Rond}
\address{Institut de Math\'ematiques de Luminy\\
  Facult\'e des Sciences de Luminy\\
 Case 907, 163 av. de Luminy\\
13288 Marseille Cedex 9, France\\} 

\begin{abstract}
We prove a \L ojasiewicz type inequality for a system of polynomial equations with coefficients in the ring of formal power series in two variables. This result is an effective version of the Strong Artin Approximation Theorem. From this result we deduce a bound of Artin functions of isolated singularities. 
\end{abstract}

\maketitle

\section{Introduction}

Let $(A,\m)$ be a  Noetherian complete local ring. The powers of the maximal ideal of $A$ defines a metric topology on $A$, called the Krull topology, the norm being defined as
$$||z||:=e^{-\ord(z)},\ \ \forall z\in A$$
where $\ord(z):=\sup\{n\in\N\ /\ z\in\m^n\}$ for all $z\in A$, $z\neq 0$. This norm extends to $A^m$ using the max norm.\\
In this paper we are interested in inequalities relating the distance to the zero set of a polynomial map defined over $A$ to the values of this map.\\
In the case $A$ is a discrete valuation ring (thus a ring of dimension 1) we have
 the following result:
\begin{theorem}\cite{Gr}\label{Gr}
Let $A$ be a complete discrete valuation ring.
Let 
$$f(z):=(f_1(z),...,f_n(z))\in A[z]^n,\ z:=(z_1,...,z_m).$$ Then there exist $a$, $b\geq 0$   such that
$$\forall c\in\N \ \forall \ovl{z}\in A^m \text{ such that } f(\ovl{z})\in\m^{ac+b}$$
$$\exists \wdt{z}\in A^m \text{ such that } f(\wdt{z})=0 \text{ and } \wdt{z}_j-\ovl{z}_j\in \m^c,\ 1\leq j\leq m.$$
\end{theorem}
The case $a=0$ in Theorem \ref{Gr} corresponds to the case $f^{-1}(0)=\emptyset$, i.e there exists a constant $b\in\N$ such that there does not exists $\wdt{z}\in A^r$ with $f(\wdt{z})\in \m^{b}$. In the case $f^{-1}(0)\neq\emptyset$, using the norm defined before it is well known that Theorem \ref{Gr} is equivalent to the following result:
\begin{thma}\label{reform} Let $A$ be a complete discrete valuation ring. Let $f(z)\in A[z]^n$ such that $f^{-1}(0)\neq\emptyset$. 
$$\exists \a\geq 0,\ C>0 \text{ s.t. } ||f(\ovl{z})||\geq Cd(f^{-1}(0),\ovl{z})^{\a} \ \ \forall \ovl{z}\in A^m$$
where $$||f(\ovl{z})||=\max_i||f_i(\ovl{z})||\ \text{ and } \ d(f^{-1}(0),\ovl{z})=\inf_{u\in A / f(u)=0}||u-\ovl{z}||.$$
\end{thma}
\begin{proof}
Let $\ovl{z}\in A^m$ and let $c\in\N$ be defined by $e^{-c}=d(f^{-1}(0),\ovl{z}))$. Then we claim that $f(\ovl{z})\notin \m^{a(c+1)+b}$. Indeed if it were not the case there would exist $\wdt{z}\in A^m$ such that $f(\wdt{z})=0$ and $\wdt{z}_j-\ovl{z}_j\in \m^{c+1}$, $1\leq j\leq m$,  and we would get 
$$d(f^{-1}(0),\ovl{z})\leq ||\wdt{z}-\ovl{z}||\leq e^{-(c+1)}<d(f^{-1}(0),\ovl{z})$$
which is not possible. Thus $||f(\ovl{z})||>e^{-(a(c+1)+b)}$, i.e. 
$$ ||f(\ovl{z})||\geq e^{-a-b+1}d(f^{-1}(0),\ovl{z})^{a} \ \ \forall \ovl{z}\in A^m.$$
Thus the inequality is satisfied with $C= e^{-a-b+1}$ and $\a=a$.\\
On the other hand, let us assume that Theorem 1.1' is satisfied. Let $\ovl{z}\in A^m$ be such that $f(\ovl{z})\in\m^{\a c+b}$ where $e^{-b}=C$. Then $||f(\ovl{z})||\leq Ce^{-\a c}$. Thus $d(f^{-1}(0),\ovl{z})\leq e^{-c}$ and there exists  $\wdt{z}\in A^m$  such that $f(\wdt{z})=0$  and  $\wdt{z}_j-\ovl{z}_j\in \m^c$, $1\leq j\leq m.$ This proves that Theorem \ref{Gr} is satisfied.

\end{proof}

This kind of inequality is true if we replace $f$ by a real analytic function on an open subset $\O$ of $\R^n$ and $A$ by a compact $K\subset \O$ (\cite{L}, see \cite{Te} for an introduction).
This kind of inequality is called a  \L ojasiewicz inequality. We are interested to extend this \L ojasiewicz inequality to the case $A$ is  a two-dimensional local complete ring or excellent Henselian local ring. We have the following analogue of Theorem \ref{Gr}:

\begin{theorem}\cite{Ar69}\cite{Po} \label{SAP} 
Let $A$ be a complete  local ring whose maximal ideal is denoted by $\m$ and let $f(z):=(f_1(z),...,f_n(z))\in A[z]^n$. Then there exists a function $\b : \N \lgw\N$ such that:
$$\forall c\in\N, \ \forall \ovl{z}_1,..., \ovl{z}_m\in A \text{ s.t. } f(\ovl{z})\in\m^{\b(c)}$$
$$\exists  \wdt{z}_1,...,\wdt{z}_m\in A \text{ s.t. } f(\wdt{z})=0 \text{ and } \wdt{z}_i-\ovl{z}_i\in\m^c  \ 1\leq i\leq m.$$
 \end{theorem}
 This theorem has been been proved by M. Artin in the case $A$ is the henselization of the ring of polynomials over a discrete valuation ring and by D. Popescu in the general case.
\begin{definition}
The least function $\b$ satisfying Theorem \ref{SAP} is called the Artin function of $f$. This is an increasing function that depends only on the ideal $I:=(f_1(z),...,f_n(z))$. See \cite{Ro06} for properties of this function.
\end{definition}
 
M. Artin raised the problem of finding estimates on the growth of Artin functions \cite{Ar70}. In general they are not bounded by affine functions as in Theorem \ref{Gr} (in \cite{Ro2} it is shown that the Artin function of $z_1^2-z_2^2z_3$ is not bounded by an affine function if $A=\k\lb x,y\rb$), thus there is no \L ojasiewicz inequality as in Theorem 1.1' in this context. But Artin's question remains widely open  in general. As Theorem \ref{Gr} is equivalent to Theorem 1.1', Theorem \ref{SAP} is equivalent to the following result:
\begin{thmb}
Let $A$ be a complete  local ring and let $f(z)\in A[z]^n$ such that $f^{-1}(0)\neq\emptyset$. Then there exists a increasing continuous function $\g :\R_{\geq 0}\lgw \R_{\geq 0}$  such that $\g(0)=0$ and 
$$||f(\ovl{z})||\geq \g\left(d(f^{-1}(0),\ovl{z})\right) \ \ \forall \ovl{z}\in A^m.$$
\end{thmb}
\begin{proof}
The proof is exactly the same as the proof that Theorem \ref{Gr} is equivalent to Theorem 1.1'. We have to replace $ac+b$ by $\b(c)$ and $\g(t):=\exp\left(-\b\left(-\ln t+1\right)\right)$ for any $t\in\ln^{-1}(\N)$. Since $\b$ may be chosen to be an increasing function, $\g$ is increasing and may be continuously extended to a continuous fnction on $\R_{\geq 0}$ . Moreover saying that $\g(0)=0$ is equivalent to say that $\b(c)$ goes to infinity as $c$ goes to infinity.

\end{proof}

The aim of this paper is to give an analogue to \L ojasiewicz  inequality when $A=\k\lb x,y\rb$ and char$(\k)=0$ (see Theorem \ref{main_theorem}). It asserts that the Artin function of $I$ is bounded by a linear function if the approximated solutions are not too close to the singular locus of $I$. This is a generalization of the main result of \cite{Ro10}, where a similar result is proven for binomials ideals. The proof is inspired by the proof of M. Artin of Theorem \ref{SAP} (see \cite{Ar69}): we use the Weierstrass division theorem in order to divide $f$ by a well chosen minor of  the Jacobian matrix of $f$ helping us to reduce the problem to the case of a system of equations with coefficients in $\k\lb x\rb$. Then we use an effective version of Theorem \ref{Gr} proven in \cite{Ro10}.  Finally we deduce from Theorem \ref{main_theorem} that the Artin function of an isolated singularity is bounded by a doubly exponential function (see Corollary \ref{de}).\\
\\
I would like to thank Michel Hickel for his comments about a previous version of this paper. I also thank the referee for its relevant comments and remarks that helped to improve the presentation.


\section{Notations}
Let $(A,\m)$ be a local ring. Let us denote by $\ord$ the $\m$-adic order on $A$,  i.e. $\ord(z):=\sup\{n\in\N\ /\ z\in\m^n\}$, where $\m$ is the maximal ideal of $A$. This order function defines a norm on $A$ as follows:
$$||z||:=e^{-\ord(z)},\ \ \forall z\in A.$$
This is an ultrametric norm, i.e. $||z+z'||\leq \max\{||z||,||z'||\}$ since $\ord(z+z')\geq \min\{\ord(z),\ord(z')\}$ for any $z$, $z'\in A$. Since $\ord(zz')\geq \ord(z)+\ord(z')$, we have $||zz'||\leq ||z||.||z'||$ for any $z$, $z'\in A$. We can extend this norm on $A^m$ by taking the maximum of the norms of the coordinates:
$$||z||:=\max_{1\leq i\leq m}||z_i||,\ \forall z:=(z_1,...,z_m)\in A^m.$$
  This norm defines a metric on $A^m$ as follows: $d(z,z'):=||z-z'||$ for any $z$, $z'\in A^m$.\\
If $a=(a_1,...,a_n)\in A^n$ and $c\in\N$, writing $a\in \m^c$ will mean $a_i\in\m^c$ for all $1\leq i\leq n$.


 \section{Jacobian ideal}
  \begin{definition}\label{Elk}\cite{Elk}
Let $A$ be a Noetherian ring and let $f_1$,..., $f_n\in A[z_1,...,z_m]$.  Let $E$ be a subset of $\lb 1,n\rb$ of cardinal $h$. We denote by $\Delta_E(f)$ the ideal of $A[z]$ generated by the $h\times h $ minors of the Jacobian matrix $\left(\frac{\partial f_i}{\partial z_j}\right)_{i\in E,  1\leq j\leq m}$ (This ideal is zero if $h>m$). We define the following ideal of $A[z]$:
  $$H_{f_1,...,f_n}:=\sum_{E}\Delta_E(f)((f_i,i\in E):I)$$
  where the sum runs over all subsets $E$ of $\lb 1,n\rb$.
  \end{definition}

    \begin{rmk}
Apparently this definition depends on the choice of the generators $f_1$,..., $f_n$ of $I$ and no details are given in \cite{Elk}. In most references using Elkik's definition nothing is said about the dependence of $H_{f_1,...,f_n}$ on the choice of the generators either it is just said that it is easy to check that it does not depend on this choice.\\
In fact, a prime ideal of $\frac{A[z]}{I}$ is in the smooth locus of the scheme $\Spec\left(\frac{A[z]}{I}\right)$ if and only if it does not contain $H_{f_1,...,f_n}\frac{A[z]}{I}$ (see for example Prop.  2.13 \cite{Sp} or Prop. 5.3 \cite{Po2}), i.e.  the smooth locus of $\Spec\left(\frac{A[z]}{I}\right)$  is $\Spec\left(\frac{A[z]}{I}\right)\backslash V\left(\sqrt{H_{f_1,...,f_n}}\frac{A[z]}{I}\right)$. In particular $\sqrt{H_{f_1,...,f_n}}\frac{A[z]}{I}$ does not depend on the presentation of the $A$-algebra $\frac{A[z]}{I}$. Another definition of an ideal containing $H_{f_1,...,f_n}$ whose support is the non-smooth locus of $\Spec\left(\frac{A[z]}{I}\right)$ and which is independent of the presentation of $\frac{A[z]}{I}$ over $A$ is given in \cite{G-R} (Definition 5.4.1).\\
Nevertheless in general  the image of $H_{f_1,...,f_n}$ in $\frac{A[z]}{I}$ depends on the generators $f_1$,..., $f_n$ as we can see in the following example:
   \end{rmk}
   
   \begin{ex}
    Set $A=\Q[x,y,z,t]$,
    $$f_1:=xz,\  f_2:=xt,\ f_3:=yz,\ f_4:=yt$$
    and let $I$ be the ideal of $A$ generated by $f_1$,...,$f_4$: $I=(x,y)\cap(z,t)$. The jacobian matrix of $f_1$,..,$f_4$ is
    $$M:=\left(\begin{array}{cccc} \frac{\partial f_1}{\partial x} &\frac{\partial f_2}{\partial x}& \frac{\partial f_3}{\partial x} &\frac{\partial f_4}{\partial x} \\
\frac{\partial f_1}{\partial y}&\frac{\partial f_2}{\partial y}&\frac{\partial f_3}{\partial y}&\frac{\partial f_4}{\partial y}\\
    \frac{\partial f_1}{\partial z}&\frac{\partial f_2}{\partial z}&\frac{\partial f_3}{\partial z}&\frac{\partial f_4}{\partial z}\\
   \frac{\partial f_1}{\partial t}&\frac{\partial f_2}{\partial t}&\frac{\partial f_3}{\partial t}&\frac{\partial f_4}{\partial t}\end{array}\right)=\left(\begin{array}{cccc} z & t & 0 &0\\
    0&0&z&t\\
    x&0&y&0\\
    0&x&0&y\end{array}\right).$$
    Then $\det(M)=xyzt-xyzt=0$. Let us compute $H_{f_1,...,f_4}$ modulo $I$:\\
    All the $3\times 3$ minors of $M$ are in $I$. Let us compute the $2\times 2$ minors of $M$ which are not in $I$. The only one  involving $f_1$ and $f_2$  is $x^2$, the only  one involving $f_3$ and $f_4$   is $y^2$. The  only one involving  $f_1$ and $f_3$ is $z^2$, the only one  involving $f_2$ and $f_4$ is $t^2$. Those involving $f_1$ and $f_4$ and 
    $f_2$ and $f_3$ are $xy$ and $zt$.    \\
    Now let us compute the ideals $((f_i,f_j):I)$ modulo $I$ for $1\leq i<j\leq 4$:
    $$((f_1,f_2):I)=(x) \text{ mod. } I$$
      $$((f_1,f_3):I)=(z) \text{ mod. } I$$
        $$((f_1,f_4):I)=((f_2,f_3):I)=(xy,zt) \text{ mod. } I$$
    $$((f_2,f_4):I)=(t) \text{ mod. } I$$
  $$((f_3,f_4):I)=(y) \text{ mod. } I$$
  Moreover the ideals $((f_i):I)=0$ modulo $I$ for any $1\leq i\leq 4$. Thus we obtain
  $$H_{f_1,...,f_4}=(x^3,y^3,z^3,t^3,(xy)^2,(zt)^2)\text{ modulo }I.$$  \\
    Now let us consider
    $$h_1:=x(z+t) ,\ h_2:=x(z-t) ,\ h_3:= yz,\ h_4= yt.$$
    These four elements generate $I$. The jacobian matrix of $h_1$,...,$h_4$ is
    $$N:=\left(\begin{array}{cccc} \frac{\partial h_1}{\partial x} &\frac{\partial h_2}{\partial x}& \frac{\partial h_3}{\partial x} &\frac{\partial h_4}{\partial x} \\
\frac{\partial h_1}{\partial y}&\frac{\partial h_2}{\partial y}&\frac{\partial h_3}{\partial y}&\frac{\partial h_4}{\partial y}\\
    \frac{\partial h_1}{\partial z}&\frac{\partial h_2}{\partial z}&\frac{\partial h_3}{\partial z}&\frac{\partial h_4}{\partial z}\\
   \frac{\partial h_1}{\partial t}&\frac{\partial h_2}{\partial t}&\frac{\partial h_3}{\partial t}&\frac{\partial h_4}{\partial t}\end{array}\right)=\left(\begin{array}{cccc} z+t & z-t & 0 &0\\
    0&0&z&t\\
    x&x&y&0\\
    x&-x&0&y\end{array}\right).$$
    Let us now compute $H_{h_1,...,h_4}$ modulo $I$:\\
    As before $\det(N)=0$ and all the $3\times 3$ minors of $N$ are in $I$. Let us compute the $2\times 2$ minors of $N$ which are not in $I$. The only one involving $h_1$ and $h_2$ is $x^2$. The only one involving $h_3$ and $h_4$ is $y^2$. The only ones involving $h_1$ and $h_3$ are $xy$ and $z(z+t)$. Those involving $h_2$ and $h_4$ are $xy$ and $t(z-t)$. Those involving $h_1$ and $h_4$ are $xy$ and $t(z+t)$ and those involving $h_2$ and $h_3$ are  $xy$ and $z(z-t)$.    \\
     Now let us compute the ideals $((h_i,h_j):I)$ modulo $I$ for $1\leq i<j\leq 4$:
    $$((h_1,h_2):I)=(x) \text{ mod. } I$$
      $$((h_1,h_3):I)=(xy,z(z+t)) \text{ mod. } I$$
        $$((h_1,h_4):I)=(xy,t(z+t)) \text{ mod. } I$$
     $$((h_2,h_3):I)=(xy,z(z-t)) \text{mod. } I$$
    $$((h_2,h_4):I)=(xy,t(z-t)) \text{ mod. } I$$
  $$((h_3,h_4):I)=(y) \text{ mod. } I$$
  Moreover  $((h_i):I)=0$ modulo $I$ for any $1\leq i\leq 4$. Thus we obtain
  $$H_{h_1,...,h_4}=(x^3,y^3,(xy)^2, z^2(z+t)^2, t^2(z+t)^2, z^2(z-t)^2, t^2(z-t)^2)\text{ modulo }I.$$  
  Clearly $H_{h_1,...,h_4}\subset H_{f_1,...,f_4}$ modulo $I$. On the other hand $z^3\in H_{f_1,...,f_4}+I$.  If $z^3\in H_{h_1,...,h_4}+I$ then $z^3\in H_{h_1,...,h_4}+I$ modulo $(x,y,t)$. But 
  $$H_{h_1,...,h_4}+I=(z^4)\text{ mod. } (x,y,t)$$
  and $z^3\notin (z^4)$. Thus $z^3\notin H_{h_1,...,h_4}+I$ and $H_{f_1,...,f_4}\neq H_{h_1,...,h_4}$ modulo $I$.\\
  \\
 In fact we can show more: let us denote by $\ovl{J}$ the integral closure of an ideal $J$. Since $H_{h_1,...,h_4}+I\subset H_{f_1,...,f_4}+I$, we have
 $\ovl{H_{h_1,...,h_4}+I}\subset \ovl{H_{f_1,...,f_4}+I}$. But since $\ovl{(z^k)}=(z^k)$ for any integer $k$, we see that 
 $$\ovl{H_{f_1,...,f_4}+I}\subsetneq \ovl{H_{h_1,...,h_4}+I}.$$   
      \end{ex}
    
         We finish this section by giving some effective bounds on $H_{f_1,...,f_n}$ that we need in the proof of Theorem \ref{main_theorem}.
    \begin{lemma}\label{bound}
    Let $I$ be an ideal of $\k\lb x,y\rb[z_1,...,z_m]$, where $x$ and $y$ are single variables, generated by polynomials $f_1$,..., $f_n$ of degree $\leq d$. Then $H_{f_1,...,f_n}$ is generated by polynomials of degree $\leq (m+2)((d+m+2)^{m+2}d)^{2^{m+1}}+(m+2)(d-1)$.
    \end{lemma}
    
    \begin{proof}
    The ideal $H_{f_1,...,f_n}$ is generated by the products of one generator of the ideal $((f_i,i\in E):I)$ and of one generator of $\D_E(f)$. If the cardinal of $E$ equals $h$, then $\D_E(f)$ is generated by polynomials of degree $\leq h(d-1)$. Moreover $((f_i,i\in E):I)$ is generated by polynomials of degree $\leq  (m+2)((d+m+2)^{m+2}d)^{2^{m+1}}$ (cf. 56 \cite{Se}). Since $h\leq m+2$ this proves the lemma.
    
        \end{proof}
        
        \begin{corollary}\label{cor}
Let $I$ be an ideal of $\k\lb x,y\rb[z_1,...,z_m]$ generated by polynomials $f_1$,..., $f_n$ of degree $\leq d$. Let $H$ be any ideal of $\k\lb x,y\rb[z]$ such that $\sqrt{H+I}=\sqrt{H_{f_1,...,f_n}+I}$. Then we have
$$(H+I)^e\subset H_{f_1,...,f_n}+I$$
where 
$$e:=\left((m+2)((d+m+2)^{m+2}d)^{2^{m+1}}+(m+2)(d-1)\right)^{\min\{n,m+1\}}.$$
        
        \end{corollary}
\begin{proof}
By Th\'eor\`eme 1 \cite{Te1} we have 
$$\sqrt{J}^{d^{\min\{n,m+1\}}}\subset J$$
 for any ideal $J$ of  $\k\lb x,y\rb[z_1,...,z_m]$ generated by $n$ polynomials of degree $\leq d$. We apply this to the ideal $J:=H_{f_1,...,f_n}+I$ using Lemma \ref{bound}.
\end{proof}

\begin{rmk}
By Proposition 2.13 \cite{Sp} or Proposition 5.3 \cite{Po2}, we can choose $H:=\sqrt{H_{f_1,...,f_n}}$ or
$$H:=\sum_{g}\sum_{E}\Delta_E(g)((g_i,i\in E):I)$$
where the first sum runs over all the sets of generators $g_1,...,g_s$ of $I$ and the second sum runs over all subsets $E$ of $\lb 1,s\rb$.\end{rmk}

\begin{rmk}\label{borne}
We claim that there exists a constant $C>1$ such that for all $d\geq 2$ and all $m\geq 1$, $e\leq d^{C^m}$.\\
Indeed, for all $d\geq 2$ and $m\geq 1$, we have 
$$e\leq (2(m+2)(d+m+2)^{m+2}d)^{2^{m+1}m}\leq (d+m+2)^{2^{m+1}m(m+5)}\leq (d+m+2)^{2^{3(m+1)}}$$
$$\leq (d+m+2)^{64^{m}}.$$
But 
$$\log\left(d+m+2\right)\leq C'^{m}\log(d)$$
for all $d\geq 2$ and $m\geq 1$ and a well chosen constant $C'>0$. Thus we set $C:=64C'$ and the claim is proven.
\end{rmk}


    \section{\L ojasiewicz inequality with respect to the Krull topology}
    
  \begin{definition}
  Let $I$ be an ideal of $A[z]$ and let $\ovl{z}\in A^m$. We say $I(\ovl{z})\in \m^{\b}$ if and only if $g(\ovl{z})\in \m^{\b}$ for all $g\in I$.\\
  The set $I^{-1}(0)$ is defined as
  $$I^{-1}(0):=\{\ovl{z}\in A^m\ / \ g(\ovl{z})=0\ \ \forall g\in I\}.$$
  \end{definition}
  Let us recall the following result that we will used in the proof of Theorem \ref{main_theorem}:
  
  \begin{theorem}\cite{Ro10}\label{1var}
For all $m,\,d\in\N$, there exists $a(m,d)\in \Z$  such that for any $f=(f_1,...,f_n)\in\k[x,z]^n$, with $z=(z_1,...,z_m)$ and $x$ a single variable, such that the total degree of $f_i$ is less or equal to $d$ for  $1\leq i\leq n$, for all $c\in\N$ and for all $z(x)\in\k[[x]]^m$ such that $f(x,z(x))\in (x)^{a(m,d)(c+1)}$, there exists $\ovl{z}(x)\in\k[[x]]^m$ such that $f(x,\ovl{z}(x))=0$ and $z(x)-\ovl{z}(x)\in (x)^c$.\\
Moreover the function $(m,d)\lgm a(m,d)$ is a polynomial function with respect to  $d$ whose degree is exponential in $m$.

\end{theorem}  

Then we can state our main theorem:
  
  \begin{theorem}\label{main_theorem}
  Let $A:=\k\lb x,y\rb$,  $x$ and $y$ being single variables, and $\k$ be an infinite field. Then there exist constants $K_1$, $K_2$, $K_3>0$  such that for any $d\geq2$ and any $m\geq 1$, 
for any ideal $I=(f_1,...,f_n)$  of $\k[x,y,z]$ generated by polynomials of degrees less than $d$ such that $f^{-1}(0)\neq \emptyset$, where $z:=(z_1,...,z_m)$, we have the following inequalities:\\
  \begin{equation}\label{eq1}\left|\left|f(\ovl{z})\right|\right|\geq (K_1d(\ovl{z},f^{-1}(0)))^{d^{\left(\frac{1}{\left|\left|H_{f_1,...,f_n}(\ovl{z})\right|\right|}\right)^{K_2m}}}\ \ \forall \ovl{z}\in A^m\backslash H_{f_1,...,f_n}^{-1}(0)\end{equation}
  \begin{equation}\label{eq2} \left|\left|f(\ovl{z})\right|\right|\geq (K_1d(\ovl{z},f^{-1}(0)))^{d^{\left(\frac{1}{\left|\left|H(\ovl{z})\right|\right|}\right)^{d^{K_3^m}}}}\ \ \forall \ovl{z}\in A^m\backslash H^{-1}(0)\end{equation}
  where $H$ is any ideal of $A[z]$ such that $\sqrt{H+I}=\sqrt{H_{f_1,...,f_n}+I}$.
  \end{theorem}
  
\begin{rmk}\label{rmk} 
  Both inequalities show that we have a \L ojasiewicz inequality as in Theorem 1.1' if we consider elements $\ovl{z}$ whose contact order with the singular locus of $X:=\Spec\left(\frac{A[z]}{I}\right)$ (i.e. $\ord(H(\ovl{z}))$ where $H$ is an ideal defining the singular locus of $X$)  is bounded.
    \end{rmk}
    
    \begin{rmk}
    In Theorem \ref{main_theorem}, both inequalities are valid only when $\ovl{z}\notin H^{-1}(0)$. In general we can do the following:\\
    Let $e$ be an integer such that $\sqrt{I}^e\subset I$. Let $f_1$,..., $f_n$ be generators of $I$ and $g_1$,..., $g_l$ be generators of $\sqrt{I}$. For any $\ovl{z}\in A^n$ and for any $i$ we have $g_i^e(\ovl{z})\in (f_1(\ovl{z}),...,f_n(\ovl{z}))$. Thus for any $i$, there exist $a_{i,1}$,..., $a_{i,n}\in A$ such that 
    $$g_i^e(\ovl{z})=a_{i,1}f_1(\ovl{z})+\cdots+a_{i,n}f_n(\ovl{z}).$$
    Hence $\ord(g_i^e(\ovl{z}))\geq \min_j\ord(f_j(\ovl{z}))$ and  $||g(\ovl{z})||^e\leq ||f(\ovl{z})||$. Since $f^{-1}(0)=g^{-1}(0)$ we see that if  $g_1$,..., $g_l$ satisfy Inequality (\ref{eq1}) or (\ref{eq2}) then $f_1$,..., $f_n$ satisfy the same kind of inequality where $d(f^{-1}(0),\ovl{z})$ is relaced by $d(f^{-1}(0),\ovl{z})^e$.\\
  In particular if  $I$ is a radical ideal and if we set $X:=\Spec\left(\frac{A}{I}\right)$, then $H$ defines the singular locus of $X$ denoted by Sing$(X)$ which is a proper closed subset of $X$. Then we have a natural stratification of $X$ where the first stratum is Reg$(X)$, the regular locus of $X$, the second one is Reg(Sing$(X))$, the third one is Reg(Sing(Sing($X)))$, etc. On each of these strata, we can apply Theorem \ref{main_theorem}. Thus we see that we can stratify $X$ into a finite set of locally closed subsets of $X$, such that on each stratum $S$ we have an inequality of the form
  $$||f(\ovl{z})||\geq (K_1d(f^{-1}(0),\ovl{z}))^{K_2^{\left(\frac{1}{\left|\left|H(\ovl{z})\right|\right|}\right)^{K_3}}} \ \ \ \forall \ovl{z}\in S$$
where $\ovl{S}=f^{-1}(0)$ ($\ovl{S}$ denotes the Zariski closure of $S$) and $S=\ovl{S}\backslash H^{-1}(0)$.\\
By replacing $H$ by some power of $H$, we may even assume that $K_3=1$.

    \end{rmk}

    \begin{rmk}
    If $f_1$,..., $f_n$ are polynomials of degree $1$ with respect to $z$, then the Artin function of $f$ is bounded by an affine function (cf. Th\'eor\`eme 3.1 \cite{Ro06}). Thus such a system satisfies Theorem 1.1'.
    \end{rmk}

    \begin{rmk}
    This theorem is not true if $A$ is of dimension more than 2. In \cite{Ro06} the following example is given:\\
    Let $A:=\k\lb x_1,x_2,x_3\rb$ and let $f:=z_1z_2-z_3z_4$. Here $H_{f}=\sqrt{H_f}=(z_1,z_2,z_3,z_4)$. For any $c\geq 3$, let us denote
    $$\ovl{z}_1^{(c)}:=x_1^c,\ \  \ovl{z}_2^{(c)}:=x_2^c,\ \ \ovl{z}_3^{(c)}:=x_1x_2-x_3^c.$$
    Then there exists $\ovl{z}_4^{(c)}\in A$ such that $\ovl{z}^{(c)}_1\ovl{z}^{(c)}_2-\ovl{z}^{(c)}_3\ovl{z}^{(c)}_4\in \m^{c^2}$. Moreover it is proved in \cite{Ro06}, that any solution $\wdt{z}\in A^4$ of $f=0$ satisfies $\displaystyle\min_{i=1,...,4}\{\ord(\ovl{z}_i^{(c)}-\wdt{z}_i)\}\leq c$. Thus there do not exist constants $a>0$ and $b>0$ such that $||f(\ovl{z}^{(c)})||\geq ad(\ovl{z}^{(c)},f^{-1}(0))^b$ for all $c\in\N$, but $||H_f(\ovl{z}^{(c)})||=e^{-2}$ is constant for any $c$.\\
    
    \end{rmk}
    
    \begin{rmk}
    This theorem is still valid if  $A$ is any excellent Henselian local ring whose completion is $\k\lb x,y\rb$ by Artin Approximation Theorem \cite{Po}. Indeed in this case the zero set of $f$ in $A$ is dense in the zero set of $f$ in $\k\lb x,y\rb$ for the topology induced by the norm $||.||$.
    \end{rmk}

  \begin{proof}[Proof of Theorem \ref{main_theorem}]
  We begin to prove the first inequality. Let us denote by $\m:=(x,y)$ the maximal ideal of $\k\lb x,y\rb$.
Let $c\in\N$. Let $s\in\N$ and let $\ovl{z}\in\k\lb x,y\rb^m$ such that $f(\ovl{z})\in \m^{\g}$ for all $f\in I$ with  
$$\g=\g(m,d,s,c):=a\big(2(m+1)s,4mds\big)(c+2s+1)$$ where $a(.,.)$ is the function of Theorem \ref{1var} and let us assume that  $H_{f_1,...,f_n}(\ovl{z})\not\subset \m^{s}$. Since $H_{f_1,...,f_n}$ is generated by the elements $\d_Ek_E$ where $\d_E$ is a minor of the Jacobian matrix $\left(\frac{\partial f_i}{\partial z_j}\right)_{i\in E,  1\leq j\leq m}$ and $k_E\in (f_i,i\in E):I)$, there exists $E\subset \lb 1,n\rb$ such that $\d_E(\ovl{z})k_E(\ovl{z})\notin \m^{s}$. In particular $\d_E(\ovl{z})\notin \m^s$. Let $\d$ denote this minor, i.e. $\d:=\d_E$. Then we remark that $\deg(\d)\leq m(d-1)$.\\
For convenience we will assume that $E=\{1,...,q\}$ where $q\leq n$.\\
\\
Let $r:=\ord(\d^2(\ovl{z}))\leq 2(s-1)$. In this case $\ord(k_E(\ovl{z}))< s-\frac{r}{2}$. If $r=0$ then
  $\d^2(\ovl{z})$ is invertible and $f(\ovl{z})\in(\d^2(\ovl{z}))\m^{\g}\subset(\d^2(\ovl{z}))\m^{c}$. In this case we set $\ovl{\ovl{z}}:=\ovl{z}$. Then let us assume that $\d^2(\ovl{z})$ is not invertible. Since $\k$ is infinite, by making a linear change of variables in $x$ and $y$, we may assume that $\d^2(\ovl{z})$ is regular with respect to $y$ and by the Weierstrass Preparation Theorem 
$\d^2(\ovl{z})=\ovl{u}\, \ovl{a}$ where $\ovl{u}$ is a unit and
$$\ovl{a}:=y^r+\ovl{a}_1(x)y^{r-1}+\cdots+\ovl{a}_r(x)$$
where $a_i(x)\in(x)^{r-i}\k\lb x\rb$, $1\leq i\leq r$.\\
Then we  perform the Weierstrass division of $\ovl{z}_i$ by $\ovl{a}$:
$$\ovl{z}_i=\ovl{a}\, \ovl{w}_i+\sum_{j=0}^{r-1}\ovl{z}_{i,j}(x)y^j$$
for $1\leq i\leq m$. Set
$$\ovl{z}^*_i:=\sum_{j=0}^{r-1}\ovl{z}_{i,j}(x)y^j,\ \ 1\leq i\leq m.$$
Then $\d^2(\ovl{z})=\d^2(\ovl{z}^*)$ mod. $\ovl{a}$ and $f_k(\ovl{z})=f_k(\ovl{z}^*)$ mod. $\ovl{a}$ for $1\leq k\leq n$.\\
\\
Let $z_{i,j}$, $1\leq i\leq m$, $0\leq j\leq r-1$, be new variables. Let us define $\displaystyle z^*_i:=\sum_{j=0}^{r-1}z_{i,j}y^j$, $1\leq i\leq m$, and
$$A(a_i,y):=y^r+a_1y^{r-1}+\cdots+a_r\in\k[y,a_1,...,a_r]$$
where $a_1$,..., $a_r$ are new variables. Then  the Euclidean division of $\d^2(z^*)$ and $f_i(z^*)$ by $A$ (seen as a polynomial in $y$)  may be written as follows:
$$\d^2(z^*)=A.Q+\sum_{l=0}^{r-1}G_ly^l$$
$$f_k(z^*)=A.Q_k+\sum_{l=0}^{r-1}F_{k,l}y^l,\ \ 1\leq k\leq r$$
 where $Q$, $Q_k\in\k[x,y, z_{i,j}, a_p]$ and $G_l$, $F_{k,l}\in\k[x,z_{i,j},a_p]$. Moreover $\deg(f_k(z^*))\leq dr$ and $\deg(\d^2(z^*))\leq 2m(d-1)r$, hence we get  $\deg(F_{k,l})\leq dr-l\leq 2ds$ and $\deg(G_l)\leq 2m(d-1)r-l\leq 4mds$ by the following lemma:
 \begin{lemma}\label{lemme1}
 Let $P(A,U, V)\in\k[A,U,V]$  where $U=(U_1,...,\,U_p)$, $A=(A_1,...,\,A_{r})$ and $V$ is a single variable, and set $A(V):=V^r+A_{1}V^{r-1}+\cdots+A_r\in \k[A,\,V]$. Let us consider the division of $P$ by $A$ with respect to $V$: $P=AQ+R$ with $\deg_{V}(R)<\deg_{V}(P)$. Then $\deg(R)\leq \deg(P)$.
\end{lemma}
\begin{proof}[Proof of Lemma \ref{lemme1}]
We can write
$P(V):=P_eV^e+\dots+P_0$ with $P_i\in \k[U]$, $P_e\neq 0$ and $\deg(P_i)\leq d-i$ where $d:=\deg(P)$. Then we have:
$$P=P_eV^{e-r}A(V)+R_1$$
with $$R_1:=(P_{e-1}-P_eA_{1})V^{e-1}+\cdots+(P_{e-r}-P_eA_r)V^{e-r}+P_{e-r-1}V^{e-r-1}+\cdots+P_0$$
where $\deg_V(R_1)<\deg_V(P)$. Moreover we see that $\deg(R_1)\leq d$. Thus we obtain the result by induction on $e:=\deg_V(P)$.
\end{proof}
 Then we have 
 $$\d^2(\ovl{z}^*)=\sum_{l=0}^{r-1}G_l(t,\ovl{z}_{i,j}(x),\ovl{a}_p(x))y^l \text{ mod. } (\ovl{a})$$
 $$f_k(\ovl{z}^*)=\sum_{l=0}^{r-1}F_{k,l}(t,\ovl{z}_{i,j}(x),\ovl{a}_p(x))y^l \text{ mod. } (\ovl{a}),\ \ 1\leq k\leq r.$$
 But $\d^2(\ovl{z}^*)=0$ mod. $(\ovl{a})$ thus 
  $G_l(x,\ovl{z}_{i,j}(x),\ovl{a}_p(x))=0$ for all $l$. Moreover $f_k(\ovl{z})=0$ mod. $(\ovl{a})+\m^{\g}$, thus
  $f_k(\ovl{z}^*)=0$ mod. $(\ovl{a})+\m^{\g}$ and $F_{k,l}(x,\ovl{z}_{i,j}(x),\ovl{a}_p(x))\in (x)^{\g}$ for all $k$ and $l$ by Remark 6.6 \cite{BM}.\\
  \\
    By Theorem \ref{1var}, there exist $\ovl{\ovl{z}}_{i,j}(x)\in\k\lb x\rb$ and $\ovl{\ovl{a}}_p(x)\in\k\lb x\rb$ for all $i$, $j$ and $p$, such that $G_l(x,\ovl{\ovl{z}}_{i,j}(x),\ovl{\ovl{a}}_p(x))=0$ and $F_{k,l}(x,\ovl{\ovl{z}}_{i,j}(x),\ovl{\ovl{a}}_p(x))=0$ for all $k$ and $l$ and $\ovl{z}_{i,j}(x)-\ovl{\ovl{z}}_{i,j}(x)$, $\ovl{a}_p(x)-\ovl{\ovl{a}}_p(x)\in(x)^{c+2s}$ for all $i$, $j$ and $p$.\\
\\
Let us denote
$$\ovl{\ovl{a}}:=y^r+\ovl{\ovl{a}}_1(x)y^{r-1}+\cdots+\ovl{\ovl{a}}_r(x)$$
$$\ovl{\ovl{z}}_i:=\ovl{\ovl{a}}\,\ovl{w}_i+\sum_{j=0}^{d-1}\ovl{\ovl{z}}_{i,j}(x)y^j$$
 for all $i$. It is straightforward to check that $f_i(\ovl{\ovl{z}})=0$ mod. $\d^2(\ovl{\ovl{z}})$  for $1\leq i\leq r$ and $\ovl{z}_j(x)-\ovl{\ovl{z}}_j(x)\in\m^{c+2s}$ for $1\leq j\leq m$.\\
\\
Since $\ord(\d^2(\ovl{\ovl{z}}))=r$, then we have $f(\ovl{\ovl{z}})\in \d^2(\ovl{\ovl{z}})\m^{c+2s-r}$. In any case we have $f(\ovl{\ovl{z}})\in\d^2(\ovl{\ovl{z}})\m^c$.
 Then we use the following generalization of the Implicit Function Theorem: \\
\begin{theorem}\cite{To}\label{To}
 Let $(A,\m_A)$ be a  complete local ring and let $f_1(z)$,..., $f_q(z)\in A[z]$ with $q\leq m$.  Let $\d$ be  a $q\times q$ minor of the Jacobian matrix $\frac{\partial(f_1,...,f_q)}{\partial(z_1,...,z_m)}$. Let us assume that there exists $\ovl{\ovl{z}}:=(\ovl{\ovl{z}}_1,...,\ovl{\ovl{z}}_m)\in A^m$ such that
$$f_i(\ovl{\ovl{z}})\in (\d(\ovl{\ovl{z}}))^2\m_A^c \text{ for all }1\leq i\leq q$$
and for some $c\in\N$. Then there exists $\wdt{z}=(\wdt{z}_1,...,\wdt{z}_m)\in A^m$ such that 
$$f_i(\wdt{z})=0 \text{ for all } 1\leq i\leq q,\ \text{ and }\  \wdt{z}_j-\ovl{\ovl{z}}_j\in (\d(\ovl{z}))\m_A^c \text{ for all } 1\leq j\leq m.$$\\
\end{theorem}
Thus $\ord(k_E(\wdt{z}))=\ord(k_E(\ovl{z}))$ since $\wdt{z}_j-\ovl{z}_j\in \m^{c+2s-r}$, $1\leq j\leq m$,  and since $\ord(k_E(\ovl{z}))< s-\frac{r}{2}\leq 2s-r$, hence $k_E(\wdt{z})\neq0$. Since $k_E(z)f_i(z)\in (f_1,...,f_q)$ for any $i$, we have $f_i(\wdt{z})=0$ for $1\leq i\leq n$.

Thus we have proved that if $f(\ovl{z})\in\m^{\g}$ and $H_{f_1,...,f_n}(\ovl{z})\not\subset \m^s$, then there exists $\wdt{z}\in A^m$ such that $f(\wdt{z})=0$ and $\wdt{z}-\ovl{z}\in\m^c$. By Theorem \ref{1var}, there exists $K$ such that $a(m,d)\leq d^{K^m}$ for all $d\geq 2$, $m\geq 1$. Since $c+2s+1\leq(c+1)(2s+1)$ we have $\g\leq b(c+1)$ where 
\begin{equation}\label{bound}b\leq (2s+1)(4mds)^{K^{2(m+1)s}}\leq d^{{K'}^{ms}}, \ \forall d\geq 2,\forall m,s\geq 1,\end{equation}
for some constant $K'$ large enough (the existence of $K'$ is proven exactly as the existence of $C$ in Remark \ref{borne}). Thus, for any $d\geq 2$, $m\geq 1$, $f(z)\in A[z]^n$ such that $\deg(f)\leq d$: 
\begin{equation}\label{implication}\ord(f(\ovl{z}))\geq d^{{K'}^{m\, \ord(H_{f_1,...,f_n}(\ovl{z}))}}(c+1)\Longrightarrow \max_{z\in f^{-1}(0)}\min_i\{\ord(z_i-\ovl{z}_i)\}\geq c.\end{equation}
Hence, there exists $K>0$ such that 
$$\forall \ovl{z}\in A^m\backslash H^{-1}_{f_1,...,f_n}(0)\ \ \max_{z\in f^{-1}(0)}\min_i\{\ord(z_i-\ovl{z}_i)\}\geq \frac{1}{d^{K^{m\,\ord(H_{f_1,...,f_n}(\ovl{z}))}}}\ord(f(\ovl{z}))-1.$$
Thus $$||f(\ovl{z})||\geq \left(\frac{1}{e}d(\ovl{z},f^{-1}(0))\right)^{d^{K^{m\,\ord(H_{f_1,...,f_n}(\ovl{z}))}}}$$
Since $K^{m\,\ord(H_{f_1,...,f_n}(\ovl{z}))}=||H_{f_1,...,f_n}(\ovl{z})||^{-\log(K)m}$, we have proved Inequality (\ref{eq1}) with $K_1:=\frac{1}{e}$ and $K_2:=\log(K)$.\\
\\
\\
\\
Let us prove the second inequality. Let $\ovl{z}\in A^m$ such that $H(\ovl{z})\not\subset \m^s$ and $I(\ovl{z})\subset\m^{\g'}$ where $$\g'=a\big(2(m+1)sd^{C^m},4mdsd^{C^m}\big)(c+2sd^{C^m}+1),$$ where  $C$ is the constant of Remark \ref{borne}, i.e. $\g'=\g(m,d,sd^{C^m},c)$. In particular $\sqrt{H+I}(\ovl{z})\not\subset \m^s$.
Then we have  $\sqrt{H+I}^{d^{C^m}}(\ovl{z})\not\subset \m^{sd^{C^m}}$,  thus $(H_{f_1,...,f_n}+I)(\ovl{z})\not\subset \m^{sd^{C^m}}$ since $\sqrt{H+I}^{d^{C^m}}\subset H_{f_1,...,f_n}+I$ by Corollary \ref{cor} and Remark \ref{borne}.  But   $I(\ovl{z})\subset\m^{\g'}\subset \m^{sd^{C^m}}$, then $H_{f_1,...,f_n}(\ovl{z})\not\subset \m^{sd^{C^m}}$. Thus, by the previous case (by replacing $s$ by $sd^{C^m}$), we see that there exists $\wdt{z}\in A^m$ such that $f(\wdt{z})=0$ and $\wdt{z}-\ovl{z}\in\m^c$. Moreover 
 there exists a constant $K''>0$ such that $d^{C^m}m\leq d^{K''^m}$ for all $d\geq 2$ and all $m\geq 1$, thus we have
 $$\g'\leq d^{K'^{msd^{C^m}}}\leq d^{K'^{sd^{K''^m}}}.$$
Hence, exactly as the end of the proof of Inequality \eqref{eq1}, we have 
$$||f(\ovl{z})||\geq (K_1d(\ovl{z},f^{-1}(0)))^{d^{\left(\frac{1}{\left|\left|H(\ovl{z})\right|\right|}\right)^{d^{K_3^m}}}}$$ 
for some positive constants $K_1>0$ and $K_3>0$ independent of $\ovl{z}$.

\end{proof}

\begin{ex}
Let $f(z)\in \k[z]^n$ and let us assume that $z_m\in \sqrt{H_{f}}$. Let $h\in(x,y)^2\k\lb x,y\rb$ be a non-zero  power series without multiple factor and let $$g(z_1,...,z_{m-1}):=f(z_1,...,z_{m-1},h).$$ Then we claim that the Artin function of $g$ is bounded by an affine function.\\
\\
Indeed, let us denote by $e$ the order of $h$. By Theorem \ref{main_theorem} there exists two constants $a\geq0$ and $b\geq 0$ (depending on $e$ and a power of $z_m$ which belongs to $H_f$) such that for any $\ovl{z}_1$,..., $\ovl{z}_{m-1}\in\k\lb x,y\rb$ with $$f(\ovl{z}_1,...,\ovl{z}_{m-1},h)\in (x,y)^{ac+b}$$
there exists $\wdt{z}_1$,...., $\wdt{z}_{m-1}$, $\wdt{h}\in\k\lb x,y\rb$ such that
$$f(\wdt{z}_1,...,\wdt{z}_{m-1},\wdt{h})=0$$
$$\text{and } \wdt{z}_i-\ovl{z}_i, \wdt{h}-h\in(x,y)^c, \ \ \forall i.$$
We claim that $J:=\left(\frac{\partial h}{\partial x},\frac{\partial h}{\partial y}\right)$, the Jacobian ideal of $h$, contains a power of $(x,y)$. Indeed, it is well known that $h\in\sqrt{J}$ (see Theorem 7.1.5 \cite{HS} for example) and after a linear change of coordinates we can write $h=g.u$ where $g$ is a Weierstrass polynomial in the variable $y$ and $u$ is a unit. Since $u$ is a unit and $J$ does not depend on the choice of cooordinates, we have $\sqrt{J}=\sqrt{J'}$ where $J'$ is the jacobian ideal of $g$. Moreover $g$ and $\frac{\partial g}{\partial y}$ are coprime since $g$ is a polynomial with no multiple factor and char$(\k)=0$. But $g\in\sqrt{J'}$ thus height$(\sqrt{J'})=2$ and $\sqrt{J}=\sqrt{J'}=(x,y)$ since $(x,y)$ is the only height two radical ideal of $\k\lb x,y\rb$. This proves that $J$ contains a power of $(x,y)$.\\
 Thus there exists $k\in\N$ such that $(x,y)^{k}\subset J^2$. Let us assume that $c>k+1$ and let us consider the equation:
$$P(x,y,x_1,y_1):=\wdt{h}(x,y)-h(x_1,y_1)=0$$
 where $x_1$ and $y_1$ are new variables. Then 
 \begin{equation*}\begin{split} P(x,y,x,y)\in (x,y)^{c}\subset (x,y)^{c-k}&\left(\frac{\partial h}{\partial x}(x,y  ),\frac{\partial h}{\partial y}(x,y)\right)^2=\\
 &=(x,y)^{c-k}\left(\frac{\partial P}{\partial x_1}(x,y,x,y),\frac{\partial P}{\partial y_1}(x,y,x,y)\right)^2\end{split}\end{equation*}
 since $\frac{\partial P}{\partial x_1}(x,y,x,y)=-\frac{\partial h}{\partial x}(x,y)$ and $\frac{\partial P}{\partial y_1}(x,y,x,y)=-\frac{\partial h}{\partial y}(x,y).$
Thus by Theorem \ref{To}, there exists $\ovl{x}_1$, $\ovl{y}_1\in \k\lb x,y\rb$ such that $P(x,y,\ovl{x}_1,\ovl{y}_1)=0$ and $\ovl{x}_1(x,y)-x$, $\ovl{y}_1(x,y)-y\in (x,y)^{c-k}$. Thus the $\k$-morphism $\phi$ defined by $\phi(p(x,y)):=p(\ovl{x}_1(x,y),\ovl{y}_1(x,y))$, for all $p\in\k\lb x,y\rb$, is a $\k$-automorphism of $\k\lb x,y\rb$. By assumption $\phi(h)=\wdt{h}$ and $\phi(p)-p\in (x,y)^{c-k}$ for any $p\in\k\lb x,y\rb$. Let $\wdt{z}_i':=\phi^{-1}(\wdt{z}_i)$ for $1\leq i\leq m$. Then $\wdt{z}_i'-z_i\in (x,y)^{c-k}$, for $1\leq i\leq m$, and $f(\wdt{z}_1',...,\wdt{z}_m',h)=0$. This proves that Artin function of $g$ is bounded by $c\lgm a(c+k)+b$.

\end{ex}

From theorem \ref{main_theorem} we can find the following bound of the Artin function of an isolated singularity (a similar bound has been given in \cite{Ro10} for binomial ideals):

\begin{corollary}\label{de}
Let   $\k$ be an infinite field. Let $I=(f_1,...,f_n)$ be an ideal of $\k[x,y,z_1,...,z_m]$ such that $I\subset (z_1,...,z_m)$. Let us assume that $H_{f_1,...,f_n}$ contains a power of the ideal $(z_1,..., z_m)$. Then the Artin function of $I$ is bounded by a function of the form $c\lgm K^{K^c}$ for some constant $K>0$.
\end{corollary}

\begin{proof}
Let $A$ denote the ring $\k\lb x,y\rb$. 
For $\ovl{z}\in A^m$  let us set  $D(\ovl{z}):=\min_i\ord(\ovl{z}_i)$. Let $k\in\N$ such that $(z)^k\subset H_{f_1,...,f_m}$ and let $d\in\N$ be a bound of the degrees of the $f_i$'s.\\
 By Theorem \ref{main_theorem}, for any $\ovl{z}\in A^m$ such that $D(\ovl{z})<\infty$,  if $I(\ovl{z})\in\m^{d^{K_1^{mk(D(\ovl{z})+1)}}(c+1)}$, then there exists $\wdt{z}\in A^m$ such that $I(\wdt{z})=0$ and $\wdt{z}-\ovl{z}\in \m^cA^m$  (see Inequality (\ref{bound}) in the proof of Theorem \ref{main_theorem}).
 Let us remark that there  exists a constant $K>1$ such that $d^{K_1^{mkc}}(c+1)\leq K^{K^c}$ for any $c\geq 1$ (this inequality is proven exactly as the existence of $C$ in Remark \ref{borne}).\\
 Now let $\ovl{z}$ be any element of $A^m$ such that $I(\ovl{z})\in \m^{K^{K^c}}$. Then two cases may occur: either $D(\ovl{z})\geq c$, either $D(\ovl{z})<c$. In the first case let us set $\wdt{z}_j:=0$ for $1\leq j\leq m$. If $D(\ovl{z})<c$, then $I(\ovl{z})\in \m^{d^{K_1^{mk(D(\ovl{z})+1)}}(c+1)}$ by assumption on $K$, thus there exists $\wdt{z}\in A^m$ such that $I(\wdt{z})=0$ and $\wdt{z}_j-\ovl{z}_j\in \m^c$ for $1\leq j\leq m$. This proves the corollary.
\end{proof}


 \end{document}